\documentclass[twoside,10pt]{amsart}
\usepackage{amsthm,amsmath,amsfonts,epsfig,amssymb}
\usepackage[T1]{fontenc}
\usepackage{graphicx}
\usepackage{enumerate}
\usepackage{xy}
\usepackage{graphics}
\xyoption{all}
\entrymodifiers={+!!<0pt,\fontdimen22\textfont2>}

\newtheorem{thm}{Theorem}[section]

\newtheorem{prop} [thm]{Proposition}
\newtheorem{lem} [thm]{Lemma}
\newtheorem{cor}[thm]{Corollary}

\newtheorem*{thm*}{Theorem}
\newtheorem*{quest*}{Question}
\theoremstyle{definition}
\newtheorem{defin}[thm]{Definition}
\theoremstyle{remark}
\newtheorem{rem}[thm]{Remark}

\def\Z{\mathbb{Z}}
\def\N{\mathbb{N}}

\def\A{\mathbb {A}}
\def\O{\mathcal O}

\def\homm#1#2#3{\mathrm{Hom}_{#1}(#2,#3)}

\def\spec#1{\text{Spec}(#1)}

\def\Ko#1{\widetilde K_0Sp(#1)}

\def\mul#1#2{\mu_{l^{#1}}^{\otimes #2}}
\def\cha#1{\mathrm{char}(#1)}

\title{On Stably Free Modules over Affine Algebras}
\author{Jean Fasel}
\author{Ravi A. Rao}
\author{Richard G. Swan}


\begin{document}

\begin{abstract}
If $X$ is a smooth affine variety of 
dimension $d$ over an algebraically closed field $k$, and if $(d-1)!\in k^\times$ 
then any stably trivial vector bundle of rank $(d-1)$ over $X$ is trivial. 
The hypothesis that $X$ is smooth can be weakened to $X$ is normal
if $d \geq 4$. 
\end{abstract}

\maketitle

\tableofcontents

\section{Introduction}

Let $R$ be an affine algebra of dimension $d$ over a field $k$. Recall that a projective $R$-module $P$ of rank $r$ is said to be \emph{cancellative} if for any projective module $Q$ such that $P\oplus R^n\simeq Q\oplus R^n$ then $P\simeq Q$. A famous theorem by Bass and Schanuel asserts that any $P$ is cancellative if $r>d$ (\cite[Theorem 9.3]{Bass64}). In general, this is the best possible bound as shown by the well known example of the tangent bundle of the real algebraic sphere of dimension $2$.

However, A.A. Suslin has shown in \cite{Suslin2}, \cite{Suslin1} that if $k$ is algebraically closed, then any projective module $P$ of rank $d$ is cancellative. In \cite{Suslin5}, he further showed that if $\mathrm{c.d.}(k)\leq 1$ then $R^d$ is cancellative provided $(d-1)!\in k^\times$, thus proving that stably free modules of rank $\geq d$ are actually free in this setting. As a consequence of this result, S.M. Bhatwadekar showed in \cite{Bhatwadekar} that any projective module $P$ of rank $d$ is cancellative in this situation. 

The natural question is to know under which conditions projective $R$-modules of rank $r<d$ are cancellative. A weaker question, asked by Suslin in \cite{Suslin1}, is to understand the conditions under which $R^r$ itself is cancellative, i.e. when stably free modules of rank $r$ are free. 

In the same article, A.A. Suslin mentions that all stably free non free modules known to him are of rank $\leq (\dim R- 1)/2$. In response to this, in  \cite{mohan} N. Mohan Kumar gave a method of constructing a non-free stably free projective module of rank $d-1$ over the coordinate ring of certain basic open sets $D(f)$ of the affine space $\A^d_k$, for $d \geq 4$, where $k$ is of transcendence degree one over an algebraically closed field. Clearing ``denominators'' he could get a smooth, rational affine variety of dimension $d$ over an algebraically closed field and a rank $d-2$ stably free non-free projective module over it. 

Thus the above question boils down to understand if stably free modules of rank $d-1$ over $R$ are free. In \cite{Fa09} the first named author made the initial breakthrough to Suslin's question by showing that over a smooth threefold $R$, any stably free rank two projective module is free if $6\in R^\times$. In this article, we extend this result and give a final answer to Suslin's question (Theorem \ref{faselswr} in the text).

\begin{thm*}
Let $R$ be a $d$-dimensional normal affine algebra over an algebraically closed field $k$ such that $\gcd((d-1)!, \cha k)=1$. If $d=3$, suppose moreover that $R$ is smooth. Then every stably free $R$-module $P$ of rank $d-1$ is free.
\end{thm*}

The solution of this problem of Suslin is an important bridge between classical algebraic K-theory, higher Grothendieck-Witt groups, and Milnor K-theory. This paper highlights these connections which will
play a critical role in future development of the subject.

We motivate our method by recollecting the salient points of Suslin's proof that stably free modules of rank $d$ are free over affine algebras of dimension $d$ over a field with $\mathrm{c.d.}(k)\leq 1$:

\begin{enumerate}[$\bullet$]
\item In view of the cancellation theorem of Bass and Schanuel, any stably free module $P$ of rank $d$ corresponds to a unimodular row $(a_1,\ldots,a_{d+1})$ of length $d+1$.
\item By Swan's Bertini theorem \cite{Swan} one can ensure, after an elementary transformation, that $C:=R/(a_4,\ldots,a_{d+1})$ is a surface and $B:=R/(a_3, \ldots, a_{d+1})$ is a smooth curve over $k$, provided $k$ is infinite (when $k$ is finite, the result follows from a result of Vaserstein). 
\item In this situation, there exists a well-defined map 
$$Um_2(B)/SL_2(B)\cap ESp(B)\to Um_{d+1}(R)/E_{d+1}(R)$$ 
under which the class of $(\overline a_1,\overline a_2)$ is sent to the class of $(a_1,\ldots,a_{d+1})$ (\cite[Corollary 2.3]{Suslin5}).
\item There is a natural bijection $Um_2(B)/SL_2(B)\cap ESp(B)\to K_1Sp(B)$ and the forgetful map $K_1Sp(B)\to SK_1(B)$ is an isomorphism (\cite[Proposition 1.7]{Suslin5}).
\item The Brown-Gersten-Quillen spectral sequence yields an isomorphism $SK_1(B)\simeq H^1(B,K_2)$, and the latter is a uniquely divisible group prime to the characteristic of $k$ (\cite[Proposition 1.4, Theorem 1.8]{Suslin5}). 
\item As a consequence of the above points, we can suppose up to elementary transformations that $(a_1,a_2,a_3,\ldots,a_{d+1})=(b_1^{d!},b_2,a_3,\ldots,a_{d+1})$. The right hand term is the first row of an invertible matrix (\cite[Theorem 2]{Suslin1}) and therefore $P$ is free.
\end{enumerate}

In \cite[Remark 2.5]{Suslin5}, Suslin observes that this method might generalize to a proof that stably free modules of rank $d-1$ are free. Indeed, our method follows closely his ideas. More precisely, we proceed as follows:
\begin{enumerate}[$\bullet$]
\item In view of Suslin's cancellation theorem, any stably free module $P$ of rank $d-1$ corresponds to a unimodular row $(a_1,\ldots,a_d)$ of length $d$.
\item Under our hypothesis, one can ensure using Swan's Bertini theorem that $B:=R/(a_4,\ldots,a_d)$ is a smooth threefold over $k$ (performing some elementary transformations if necessary).
\item There is a natural map
$$Um_3(B)/E_3(B)\to Um_d(R)/E_d(R)$$
such that $(\overline a_1,\overline a_2,\overline a_3)\mapsto (a_1,a_2,a_3,\ldots,a_d)$.
\item Using Rao-van der Kallen's result (\cite[Corollary 3.5]{Rao94}), the Vaserstein symbol 
$$V:Um_3(B)/E_3(B)\to W_E(B)$$
is a bijection, where $W_E(B)$ is the elementary symplectic Witt group as defined in \cite{SV}. Thus $Um_3(B)/E_3(B)$ is endowed with the structure of an abelian group.
\item There is an isomorphism $W_E(B)\simeq GW_{1,\mathrm{red}}^3(B)$, where $GW_{1,\mathrm{red}}^3(B)$ is a reduced Grothendieck-Witt group of $B$.
\item Using the Gersten-Grothendieck-Witt spectral sequence, the group $GW_{1,\mathrm{red}}^3(B)$ is proved to be isomorphic to $H^2(B,K_3)$, and the latter is shown to be divisible using the Merkurjev-Suslin theorem on $K_2^M$ (\cite{Merkusus}). 
\item As $\sqrt {-1}\in k$, we can use Rao's antipodal Lemma \cite[Lemma 1.3.1]{Rao87} and the above points to suppose that $(a_1,a_2,a_3,\ldots,a_d)=(b_1^{(d-1)!},b_2,a_3,\ldots,a_d)$ up to elementary transformations. This allows to conclude as in Suslin's proof.
\end{enumerate}

The organization of the paper is as follows. Section \ref{gw} gives the basic definitions and results of the theory of Grothendieck-Witt groups as developed by M. Schlichting. In Section \ref{elementary}, we recall the definition of the elementary symplectic Witt group $W_E$. We also recall an exact sequence, later used in Section \ref{ident} to show that the elementary symplectic Witt group is precisely a Grothendieck-Witt group (Theorem \ref{main} in the text). This is one of the main observations of this work, as it allows to use the powerful techniques of the theory of Grothendieck-Witt groups to study $W_E$. Among these techniques, one of the most useful is the Gersten-Grothendieck-Witt spectral sequence. It is the analogue of the classical Brown-Gersten-Quillen spectral sequence in $K$-theory. For a smooth affine threefold $R$ over an algebraically closed field, this sequence is easy to analyse and allows to prove that $W_E(R)\simeq H^2(R,K_3)$. In Section \ref{div}, we show that this cohomology group is divisible. We claim no originality in this computation, which seems to be well-known to specialists. All the pieces fall together in Section \ref{principal} where we finally prove our main theorem, following the path described above. As a consequence, we derive from the main result a few prestabilization results for the functor $K_1$.  

It must be observed that the question whether \emph{every} projective module of rank $d-1$ over an algebra of dimension $d$ over an algebraically closed field $k$ is cancellative is still open. As opposed to the case of projective modules of rank $d$, the full result is not an immediate consequence of the fact that $R^{d-1}$ is cancellative. We hope to treat this question somewhere in further work.

\subsection{Conventions} Every ring is commutative with $1$, any algebra is of
finite type over some field. For any abelian group $G$ and any $n\in\N$, we denote by $\{n\}G$ the subgroup of $n$-torsion elements in $G$. All the fields considered are of
characteristic different from $2$. If $X$ is a scheme and $x_p\in
X^{(p)}$, we denote by $\mathfrak m_p$ the maximal ideal in
$\O_{X,x_p}$ and by $k(x_p)$ its residue field. Finally $\omega_{x_p}$
will denote the dual of the $k(x_p)$-vector space $\wedge^p(\mathfrak
m_p/\mathfrak m_p^2)$ (which is one-dimensional if $X$ is regular at
$x_p$). If $M$ and $N$ are square matrices, we denote by $M\perp N$ the matrix $\begin{pmatrix} M & 0 \\ 0 & N\end{pmatrix}$.

\subsection{Acknowledgements}
The first author wishes to thank the Swiss National Science Foundation, grant $\mathrm{PP00P2}\_129089/1$, for support.

\section{Grothendieck-Witt groups}\label{gw}

\subsection{Preliminaries}\label{gwgroups}
In this section, we recall a few basic facts about Grothendieck-Witt groups. These are a modern version of Hermitian $K$-theory. The general reference here is the work of M. Schlichting (\cite{Schlichting09}, \cite{Schlichting09bis}). Since we use Grothendieck-Witt groups only for affine schemes, we restrict to this case.

Let $R$ be a ring with $2\in R^{\times}$. Let $\mathcal P(R)$ be the category of finitely generated projective $R$-modules and $Ch^b(R)$ be the category of bounded complexes of objects in $\mathcal P(R)$. It carries the structure of an exact category, by saying that an exact sequence of complexes is exact if it is exact in $\mathcal P(R)$ degreewise. For any line bundle $L$ on $R$, the duality $\homm R{\_}L$ on $\mathcal P(R)$ induces a duality $\sharp_L$ on $Ch^b(R)$ and the canonical identification of a projective module with its double dual gives a natural isomorphism of functors $\varpi_L:1\to \sharp_L\sharp_L$. One can also define a weak-equivalence in $Ch^b(R)$ to be a quasi-isomorphism of complexes. This shows that $(Ch^b(R),qis,\sharp_L,\varpi_L)$ is an \emph{exact category with weak-equivalences and duality} in the sense of \cite[\S 2.3]{Schlichting09bis} (see also [loc. cit., \S 6.1]). The translation functor (to the left) $T:Ch^b(R)\to Ch^b(R)$ yields new dualities $\sharp^n_L:=T^n\circ \sharp_L$ and canonical isomorphisms $\varpi^n_L:=(-1)^{n(n+1)/2}\varpi_L$. 

To any exact category with weak-equivalences and duality, Schlichting associates a space $\mathcal{GW}$ and defines the (higher) Grothendieck-Witt groups to be the homotopy groups of that space (\cite[\S 2.11]{Schlichting09bis}). More precisely:

\begin{defin} 
We denote by $GW_i^j(R,L)$ the group $\pi_i\mathcal{GW}(Ch^b(R),qis,\sharp^j_L,\varpi^j_L)$. If $L=R$, we simply put $GW_i^j(R)=GW_i^j(R,R)$.
\end{defin}

Of course, the Grothendieck-Witt groups coincide with Hermitian $K$-theory (\cite{Karoubi}, \cite{Karoubi80}), at least when $2\in R^\times$ (\cite[remark 4.16]{Schlichting09}, see also \cite{Horn}). In particular, $GW_i^0(R)=K_iO(R)$ and $GW_i^2(R)=K_iSp(R)$. The twisted groups coincide with Karoubi $U$ groups, i.e. $GW_i^1(R)={}_{-1}U_i(R)$ and $GW_i^3(R)=U_i(R)$.

The basic tool in the study of Grothendieck-Witt groups is the \emph{fundamental exact sequence} (\cite{Karoubi80}, \cite[Definition 3.6]{Horn}; see also \cite[Theorem 8]{FS}):
$$\xymatrix@C=1.4em{\ldots\ar[r] & GW_i^j(R)\ar[r]^-f & K_i(R)\ar[r]^-H & GW_i^{j+1}(R)\ar[r]^-\delta & GW_{i-1}^j(R)\ar[r]^-f & K_{i-1}(R)\ar[r] & \ldots}$$
The fundamental exact sequence ends with Witt groups as defined by Balmer in \cite{Balmer}, see \cite[Lemma 4.13]{Schlichting09} or \cite[Theorem 9]{FS}:
$$\xymatrix@C=1.5em{\ldots\ar[r] & GW^i(R)\ar[r]^-f & K_0(R)\ar[r]^-H & GW^{i+1}(R)\ar[r] & W^{i+1}(R)\ar[r] & 0.}$$
The homomorphism $f$ is called the \emph{forgetful} homomorphism and the homomorphism $H$ is called the \emph{hyperbolic} homomorphism.
The object of the next section is to give some basic computations of Grothendieck-Witt groups of fields that will be used later.

\subsection{Basic computations} 

\begin{lem}\label{gw1}
For any field $F$ of characteristic different from $2$, the groups $GW_i^{i+1}(F)$ are trivial for $i=0,1,2$. 
\end{lem}

\begin{proof}
First observe that by definition $GW_2^2(F)=K_2Sp(F)$. Recall from \cite[Corollary 6.4]{Suslin} that there is a homomorphism $K_2Sp(F)\to I^2(F)$ such that the following diagram is a fibre product
$$\xymatrix{K_2Sp(F)\ar[r]\ar[d]_-f & I^2(F)\ar[d] \\
K_2(F)\ar[r] & I^2(F)/I^3(F)}$$
where the bottom horizontal map is the homomorphism defined by Milnor in \cite[Theorem 4.1]{Milnor69} and the right vertical map is the quotient map.
It follows that the forgetful map $GW_2^2(F)\to K_2(F)$ is surjective and the exact sequence
$$\xymatrix{GW^2_2(F)\ar[r]^-f & K_2(F)\ar[r]^-H & GW_2^3(F)\ar[r] & GW_1^2(F)}$$
shows that it suffices to prove that $GW_1^2(F)=0$ to have $GW_2^3(F)=0$. This is clear since $GW_1^2(F)=K_1Sp(F)=0$. To conclude, we are left to prove that $GW_0^1(F)=0$. But this is clear since the sequence
$$\xymatrix{GW(F)\ar[r]^-f & K_0(F)\ar[r]^-H & GW_0^1(F)\ar[r] & W^1(F)}$$
is exact and $W^1(F)=0$ (\cite[Proposition 5.2]{BW}).
\end{proof}

\begin{lem}\label{gw2}
The hyperbolic functor $K_i(F)\to GW_i^{i+2}(F)$ is an isomorphism for $i=0,1$. 
\end{lem}

\begin{proof}
This is obvious and well known for $GW^2(F)=K_0Sp(F)$. For $GW_1^3(F)$, it suffices to use the fundamental exact sequence and Lemma \ref{gw1} 
\end{proof}

\begin{lem}\label{gw3}
The forgetful functor induces surjections $f:GW_i^{i-1}(F)\to \{2\}K_i(F)$ for $i=1,2$. Moreover, if $F$ is algebraically closed then $f:GW_1^0(F)\to \{\pm 1\}$ is an isomorphism. 
\end{lem}

\begin{proof}
For any field $F$, the fundamental exact sequence
$$\xymatrix{GW_i^{i-1}(F)\ar[r]^-f & K_i(F)\ar[r]^-H & GW_i^i(F)}$$
and the proof of Lemma \ref{gw1} yield surjective homomorphisms 
$$GW_i^{i-1}(F)\to \{2\}K_i(F)$$ 
for $i=1,2$. Suppose now that $F$ is algebraically closed. The fundamental exact sequence
$$\xymatrix@C=1.3em{GW_1^3(F)\ar[r]^-f & K_1(F)\ar[r]^-H & GW_1^0(F)\ar[r] & GW^3(F)\ar[r]^-f & K_0(F)\ar[r]^-H & GW(F)}$$
and Lemma \ref{gw2} give an exact sequence
$$\xymatrix{0\ar[r] & K_1(F)/2\ar[r] & GW_1^0(F)\ar[r] & GW^3(F)\ar[r] & 0}.$$
Since $F$ is algebraically closed, the left term is trivial. Now $GW^3(F)=\Z/2$ by \cite[Lemma 4.1]{FS2}.
\end{proof}

\section{The elementary symplectic Witt group}\label{elementary}

Let $R$ be a ring (with $2\in R^\times$). We first briefly recall the definition of the elementary symplectic Witt group considered in \cite{SV}. For any $n\in \N$, let $S^\prime_{2n}(R)$ be the set of skew-symmetric matrices in $GL_{2n}(R)$. For any $r\in \N$, Let $\psi_{2r}$ be the matrix defined inductively by 
$$\psi_2:=\begin{pmatrix}   0 & 1\\ -1 & 0\end{pmatrix}$$
and $\psi_{2r}:=\psi_2\perp \psi_{2r-2}$. Observe that $\psi_{2r}\subset S^\prime_{2r}$. For any $m<n$, there is an obvious inclusion of $S^\prime_{2m}(R)$ in $S^\prime_{2n}(R)$ defined by $G\mapsto G\perp \psi_{2n-2m}$. We define $S^\prime(R)$ as the set $\cup S^\prime_{2n}(R)$. There is an equivalence relation on $S^\prime(R)$ defined as follows:

If $G\in S^\prime_{2n}(R)$ and $G^\prime\in S^\prime_{2m}(R)$, then $G\sim G^\prime$ if and only if there exists $t\in\N$ and $E\in E_{2(m+n+t)}(R)$ such that 
$$G\perp \psi_{2(m+t)}=E^t(G^\prime\perp \psi_{2(n+t)})E.$$
Observe that $S^\prime(R)/\sim$ has the  structure of an abelian group, with $\perp$ as operation and $\psi_2$ as neutral element. One can also consider the subsets $S_{2n}(R)\subset S^\prime_{2n}(R)$ of invertible skew-symmetric matrices with Pfaffian equal to $1$. If $S(R):=\cup S_{2n}(R)$, it is easy to see that $\sim$ induces an equivalence relation on $S(R)$ and that $S(R)/\sim$ is also a group. 

\begin{defin}
We denote by $W^\prime_E(R)$ the group $S^\prime(R)/\sim$ and by $W_E(R)$ the group $S(R)/\sim$. The latter is called \emph{elementary symplectic Witt group}.
\end{defin}

\subsection{An exact sequence}
In this section, we prove a few basic facts  about the elementary symplectic Witt group and obtain a useful exact sequence (see \cite[\S 2]{Fa09}). If $G$ is a skew-symmetric invertible matrix of size $2n$, it can be seen as a skew-symmetric form on $R^{2n}$. This yields a map $\varphi_{2n}:S^\prime_{2n}(R)\to \Ko R$, where the latter is the reduced symplectic $K_0$ of the ring $R$. This map passes to the limit and preserves the equivalence relation $\sim$ (as well as the orthogonal sum $\perp$). Hence we get a homomorphism 
$$\varphi:W^\prime_E(R)\to \Ko R.$$
We define a map
$$\eta_{2n}:GL_{2n}(R)\to S^\prime(R)$$ 
by $\eta_{2n}(G)=G^t \psi_{2n}G$ for any $G\in GL_{2n}(R)$, and a map 
$$\eta_{2n+1}:GL_{2n+1}(R)\to S^\prime(R)$$  
by $\eta_{2n+1}(G)=(G\perp 1)^t \psi_{2(n+1)}(G\perp 1)$ for any $G\in GL_{2n+1}(R)$. These maps obviously pass to the limit, and we obtain a map $\eta:GL(R)\to W^\prime_E(R)$ after composing with the projection $S^\prime(R)\to W^\prime_E(R)$. Using Whitehead Lemma (\cite[proof of Lemma 2.5]{Milnor71}), it is not hard to 
see that $\eta$ induces a homomorphism
$$\eta:K_1(R)\to W^\prime_E(R).$$

\begin{prop}\label{exactII}
The sequence
$$\xymatrix{K_1Sp(R)\ar[r]^-{f^\prime} & K_1(R)\ar[r]^-\eta & W^\prime_E(R)\ar[r]^-\varphi & \Ko R\ar[r]^-f & \widetilde K_0(R)}$$
is exact, where $f^\prime$ is the homomorphism induced by $Sp_{2n}\subset SL_{2n}$ for any $n$, $f$ is the forgetful homomorphism and $\widetilde K_0(R)$ is the reduced $K_0$ of $R$.
\end{prop}

\begin{proof} 
First, a straightforward computation shows that this sequence is a complex. We now check that it is exact.

It is well known that any element of $\Ko R$ is of the form $[P,\phi]$ for some projective module $P$ and some skew-symmetric form $\phi:P\to P^\vee$ (\cite[Proposition 2]{Bass75}). But $f(P,\phi)=0$ if and only if $P$ is stably free and adding a suitable $H(R^s)$ we see that $P$ is given by a skew-symmetric matrix on some $R^{2n}$. 

Let $M$ be a skew-symmetric (invertible) matrix such that $\varphi(M)=0$. Adding if necessary some multiple of $\psi_2$, we have $M=G^t\psi_{2n}G$ for some $n\in \N$ and therefore $M=\eta(G)$. 

If $G\in K_1(R)$ is such that $\eta(G)=0$, we can suppose that $G\in GL_{2n}(R)$ for some $n\in\N$ and the triviality of $\eta(G)$ is expressed as $G^t\psi_{2n}G\sim \psi_2$. This translates as
$$\psi_{2(n+s)}=E^t((G^t\psi_{2n}G)\perp\psi_{2s})E=E^t((G\perp I_{2s})^t\psi_{2(n+s)}(G\perp I_{2s}))E$$
for some $s\in\N$ and some $E\in E_{2(n+s)}(R)$. Therefore $(G\perp I_{2s})E\in Sp_{2(n+s)}(R)$. 
\end{proof}


\section{Identification with $GW^3_1(R)$}\label{ident}

\subsection{The group $GW^3_1(R)$}
Recall from Section \ref{gwgroups} that $GW_1^3(R)$ is defined to be the group $\pi_1\mathcal{GW}(Ch^b(R),qis,\sharp^3,\varpi^3)$. 

Consider the category $\mathcal {SF}(R)$ of stably-free $R$-modules of finite rank and its subcategory $\mathcal {F}(R)$ of free $R$-modules of finite rank. Let $Ch_{\mathrm{sf}}^b(R)$ (resp. $Ch_{\mathrm{free}}^b(R)$)  be the category $Ch_{\mathrm{sf}}^b(R)$ of bounded complexes of objects of $\mathcal {SF}(R)$ (resp. $\mathcal{F}(R)$). Restricting the duality, the canonical isomorphism and the quasi-isomorphisms to $Ch_{\mathrm{sf}}^b(R)$ and $Ch_{\mathrm{free}}^b(R)$, we see that both categories are also exact categories with weak-equivalences and duality. If we denote by $GW^3_{1,{\mathrm{sf}}}(R)$ the Grothendieck-Witt group obtained from the category $Ch_{\mathrm{sf}}^b(R)$ and $GW^3_{1,{\mathrm{free}}}(R)$ the Grothendieck-Witt group obtained from the category $Ch_{\mathrm{free}}^b(R)$, then the natural map $Ch_{\mathrm{sf}}^b(R)\to Ch^b(R)$ induces an isomorphism 
$$GW^3_{1,{\mathrm{sf}}}(R)\to GW^3_1(R)$$
by \cite[Theorem 7]{Schlichting09bis}, while the natural map $Ch_{\mathrm{free}}^b(R)\to Ch_{\mathrm{sf}}^b(R)$ induces an isomorphism 
$$GW^3_{1,{\mathrm{free}}}(R)\to GW^3_{1,{\mathrm{sf}}}(R)$$
by \cite[Lemma 9]{Schlichting09bis}.

Recall moreover that the group $GW^3_1(R)$ fits in the fundamental exact sequence
$$\xymatrix{K_1Sp(R)\ar[r]^-{f^\prime} & K_1(R)\ar[r]^-H & GW^3_1(R)\ar[r]^-\delta & K_0Sp(R)\ar[r]^-f & K_0(R)}$$
which gives an exact sequence
$$\xymatrix{K_1Sp(R)\ar[r]^-{f^\prime} & K_1(R)\ar[r]^-H & GW^3_1(R)\ar[r]^-\delta & \Ko R\ar[r]^-f & \widetilde K_0(R)}$$
\subsection{The group V$(R)$}
Consider the set of triples $(L,f_0,f_1)$, where $L$ is a free module and $f_i:L\simeq L^\vee$ are skew-symmetric isomorphisms ($i=0,1$). Two triples $(L,f_0,f_1)$ and $(L^\prime,f^\prime_0,f^\prime_1)$ are isometric if there exists an isomorphism $\alpha:L\to L^\prime$ such that $f_i=\alpha^\vee f^\prime_i \alpha$ for $i=0,1$. We denote by $[L,f_0,f_1]$ the isometry class of a triple $(L,f_0,f_1)$. The orthogonal sum
$$[L,f_0,f_1]\perp [L^\prime,f^\prime_0,f^\prime_1]:=[L\oplus L^\prime, f_0\perp f^\prime_0,f_1\perp f^\prime_1]$$
endows the set MV$^\prime(R)$ of isometry classes of triples $[L,f_0,f_1]$ with a structure of a monoid and we can consider its Grothendieck group $V^\prime(R)$. We denote by $V(R)$ the quotient of $V^\prime(R)$ by the subgroup generated by the relations
$$[L,f_0,f_1]+[L,f_1,f_2]-[L,f_0,f_2].$$
The group $V(R)$ is naturally isomorphic to $GW^3_{1,{\mathrm{free}}}(R)$ by \cite{Schlichting09}, and the latter is isomorphic to $GW^3_1(R)$ as seen above.

\subsection{The isomorphism}

Following \cite[\S3]{SV}, we define for any $n\in \N$ a matrix $\sigma_{2n}:R^{2n}\to R^{2n}$ inductively by
\[
\sigma_2:=\begin{pmatrix} 0 & 1 \\ 1 & 0\end{pmatrix}
\]
and $\sigma_{2n}:=\sigma_2\perp\sigma_{2n-2}$. Observe that $\sigma_{2n}^2=Id_{2n}$.

Let $(L,f_0,f_1)$ be a triple as in the above section, with $L$ free. 
Choose an isomorphism $g:R^{2n}\to L$ and consider the isomorphism 
$$(g^\vee f_1g)\perp \sigma_{2n}(g^\vee f_0g)^{-1}\sigma_{2n}^\vee:R^{2n}\oplus (R^{2n})^\vee \to (R^{2n})^\vee\oplus R^{2n}.$$
It is clearly skew-symmetric and we can consider its class in $W^\prime_E(R)$. 

\begin{lem}\label{invariant}
The class of $(g^\vee f_1g)\perp \sigma_{2n}(g^\vee f_0g)^{-1}\sigma_{2n}^\vee$ in $W^\prime_E(R)$ does not depend on $g$.
\end{lem}

\begin{proof}
If $h:R^{2n}\to L$ is another isomorphism, let $a:= h^{-1}g\perp \sigma_{2n}^\vee(g^{-1}h)^\vee \sigma_{2n}^\vee$. Then we have
$$(g^\vee f_1g)\perp (-g^\vee f_0g)^{-1}=a^\vee [(h^\vee f_1h)\perp \sigma_{2n}(h^\vee f_0h)^{-1}\sigma_{2n}^\vee]a.$$
It is then sufficient to prove that the conjugation by $m\perp \sigma_{2n}^\vee(m^{-1})^\vee\sigma_{2n}^\vee$ is trivial in $W^\prime_E(R)$ for any automorphism $m$ of $R^{2n}$. To achieve this, we adapt the arguments of \cite[Lemma 4.1.3]{Barge08}.

Let $h_{4n}:R^{2n}\oplus (R^{2n})^\vee\to (R^{2n})^\vee\oplus R^{2n}$ be given by the matrix $\begin{pmatrix} 0 & \sigma_{2n}^\vee \\ -\sigma_{2n} & 0\end{pmatrix}$. We have 
$$\begin{pmatrix} (m^{-1})^\vee & 0\\0 & \sigma_{2n}m\sigma_{2n}\end{pmatrix}\cdot h_{4n}\cdot \begin{pmatrix} m^{-1} & 0\\ 0 & \sigma_{2n}^\vee m^\vee\sigma_{2n}^\vee \end{pmatrix}=h_{4n}.$$
Adding $h_{4n}$ to $(h^\vee f_1h)\perp \sigma_{2n}(h^\vee f_0h)^{-1}\sigma_{2n}^\vee$, we see therefore that we have to prove that $m\perp  \sigma_{2n}^\vee(m^{-1})^\vee\sigma_{2n}^\vee\perp  m^{-1}\perp \sigma_{2n}^\vee m^\vee\sigma_{2n}^\vee$ is an elementary matrix. This is clear by Whitehead Lemma (\cite[proof of Lemma 2.5]{Milnor71}).
\end{proof}

We set $\zeta(L,f_0,f_1)$ to be the class of $(g^\vee f_1g)\perp \sigma_{2n}(g^\vee f_0g)^{-1}\sigma_{2n}^\vee$ in $W^\prime_E(R)$ for any isomorphism $g:R^{2n}\to L$.
Next we deal with isometries:

\begin{lem}
If $(L,f_0,f_1)$ and $(L^\prime,f_0^\prime,f_1^\prime)$ are isometric, then we have an equality $\zeta(L,f_0,f_1)=\zeta(L^\prime,f^\prime_0,f^\prime_1)$ in $W^\prime_E(R)$.
\end{lem}

\begin{proof}
If $g:R^{2n}\to L$ is an isomorphism, and $\alpha:L\to L^\prime$ is an isometry then $\alpha g:R^{2n}\to L^\prime$ is an isomorphism and the result follows from $\alpha^\vee f_i^\prime\alpha=f_i$ for $i=0,1$.
\end{proof}

Hence we see that the class $\zeta([L,f_0,f_1])$ is well defined. It is easy to check that it respects the direct sum, and thus we obtain a homomorphism
$$\zeta:V^\prime(R)\to W^\prime_E(R).$$
\begin{prop}
The homomorphism $\zeta:V^\prime(R)\to W^\prime_E(R)$ induces a homomorphism
$$\zeta:V(R)\to W^\prime_E(R).$$
\end{prop}

\begin{proof}
By definition of $\zeta$, it suffices to check that 
$$\zeta([A^{2n},f_0,f_1])+\zeta([A^{2n},f_1,f_2])=\zeta([A^{2n},f_0,f_2])$$
in $W^\prime_E(R)$. The left-hand term is equal to
$$f_1\perp \sigma_{2n} f_0^{-1}\sigma_{2n}^\vee\perp f_2\perp \sigma_{2n}f_1^{-1}\sigma_{2n}^\vee$$
and it suffices therefore to prove that $f_1\perp \sigma_{2n}f_1^{-1}\sigma_{2n}^\vee$ is elementary equivalent to some $\psi_{2n}$ to conclude. Since $\sigma_{2n}f_1^{-1}\sigma_{2n}^\vee$ is precisely the inverse of $f_1$ in $W_E(R)$ by \cite[\S3]{SV}, we are done.
\end{proof}

It remains to prove that $\zeta$ is an isomorphism, but all the work is almost done. 

\begin{thm}\label{main}
Let $R$ be a ring with $2\in R^\times$. The following diagram commutes
$$\xymatrix{K_1(R)\ar[r]^-H\ar@{=}[d] & V(R)\ar[r]^-\delta\ar[d]_-\zeta & \widetilde K_0Sp(R)\ar@{=}[d] \\
K_1(R)\ar[r]_-\eta & W^\prime_E(R)\ar[r]_-\varphi & \widetilde K_0Sp(R)}$$
and $\zeta$ is an isomorphism.
\end{thm}

\begin{proof}
We first prove that the diagram commutes. Let $a\in K_1(R)$ be represented by a matrix $G\in {\rm GL}_{2n}(R)$. Then $H$ is given by $H(a)=[A^{2n},\psi_{2n},G^\vee\psi_{2n}G]$ (see for instance \cite{Karoubi}). Since $\sigma_{2n}(\psi_{2n})^{-1}\sigma_{2n}^\vee=\psi_{2n}$, the left hand square commutes. If $[L,f_0,f_1]$ is in $V(R)$, then $\delta([L,f_0,f_1])=[L,f_1]-[L,f_0]$ in $\Ko R$. On the other hand, 
$$
-[L,f_0]=[R^{2n},-g^\vee f_0 g]=[(R^{2n})^\vee,(g^\vee f_0 g)^{-1}]=[(R^{2n})^\vee,\sigma_{2n}(g^\vee f_0 g)^{-1}\sigma_{2n}^\vee]
$$
and the right-hand square also commutes.

It follows then from Proposition \ref{exactII} and the five lemma that $\zeta$ is an isomorphism.
\end{proof}

This theorem shows that $W^\prime_E(R)$ inherits all the good properties of $GW_1^3(R)$.


\section{The Gersten-Grothendieck-Witt spectral sequence}\label{ggw}
Let $X$ be a regular scheme. Recall from \cite[Theorem 25]{FS} that the Gersten-Grothendieck-Witt spectral sequence is a spectral sequence $E(n)$ defined for any $n\in \Z$ converging to $GW^n_{n-*}(X)$ with terms on page $1$ of the form (recall our conventions about $\omega_{x_p}$):
$$E(n)_1^{p,q}= \displaystyle{\bigoplus_{x_p\in X^{(p)}} GW_{n-p-q}^{n-p}(k(x_p),\omega_{x_p})}.$$
By construction of the corresponding spectral sequences, the forgetful homomorphism and the hyperbolic homomorphism induce morphisms of spectral sequences between the Gersten-Grothendieck-Witt spectral sequence and the Brown-Gersten-Quillen spectral sequence in $K$-theory.

\begin{prop}\label{presentation}
Let $X$ be a smooth affine threefold over an algebraically closed field $k$. Then the Gersten-Grothendieck-Witt spectral sequence $E(3)^{p,q}$ yields an isomorphism
$$W_E(X)\simeq H^2(X,K_3).$$
\end{prop}

\begin{proof}
First remark that the line $q=1$ is trivial by Lemma \ref{gw1}. The line $q=2$ is as follows:
$$\xymatrix{GW_1^3(k(X),\omega)\ar[r] & \displaystyle{\bigoplus_{x_1\in X^{(1)}}GW^2(k(x_1),\omega_{x_1}) }\ar[r] & 0}$$
Lemma \ref{gw2} shows that this is isomorphic, via $H$, to 
$$\xymatrix{K_1(k(X))\ar[r] & \displaystyle{\bigoplus_{x_1\in X^{(1)}}K_0(k(x_1))}\ar[r] & 0}$$
whose homology at degree $0$ is just $\O_X(X)^\times$. Now the Pfaffian homomorphism $GW_1^3(X)\to \O_X(X)^\times$ is clearly split, and we see that the kernel of the edge homomorphism $GW_1^3(X)\to E(3)_{\infty}^{0,2}$ is precisely $W_E(X)$. 

We now show that $E(3)_{\infty}^{2,0}\simeq H^2(X,K_3)$. But $E(3)_{\infty}^{2,0}$ is the homology of the complex
$$\xymatrix@C=1.2em{\displaystyle{\bigoplus_{x_1\in X^{(1)}}GW_2^2(k(x_1),\omega_{x_1}) }\ar[r] & \displaystyle{\bigoplus_{x_2\in X^{(2)}}GW_1^1(k(x_2),\omega_{x_2}) }\ar[r] & \displaystyle{\bigoplus_{x_3\in X^{(3)}}GW(k(x_1),\omega_{x_1}) } .}$$
We use the forgetful functor to compare this sequence with the corresponding sequence
$$\xymatrix@C=1.2em{\displaystyle{\bigoplus_{x_1\in X^{(1)}}K_2(k(x_1)) }\ar[r] & \displaystyle{\bigoplus_{x_2\in X^{(2)}}K_1(k(x_2)) }\ar[r] & \displaystyle{\bigoplus_{x_3\in X^{(3)}}K_0(k(x_1)) } }$$
in $K$-theory. After d\'evissage, the forgetful functor and the proof of Lemma \ref{gw1} yield exact sequences 
$$\xymatrix{0\ar[r] & I^{4-p}(F)\ar[r] & GW_{3-p}^{3-p}(F)\ar[r]^-f & K_{3-p}(F)\ar[r] & 0,}$$
for any field $F$ and for $p=1,2,3$. Now if $x_p\in X^{(p)}$, then $cd(k(x_p))\leq 3-p$ by \cite[\S 4.2, Proposition 11]{Serre}. Hence $H^{4-p}(k(x_p),\mu_2)=0$ and the latter is isomorphic to $I^{4-p}(k(x_p))/I^{5-p}(k(x_p))$ by \cite[Theorem 4.1]{OVV} and \cite[Theorem 7.4]{Voevodsky}. The Arason-Pfister Hauptsatz \cite{Arasonpf} then shows that $I^{4-p}(k(x_p))=0$. The forgetful homomorphism therefore induces an isomorphism of complexes between 
$$\xymatrix@C=1.2em{\displaystyle{\bigoplus_{x_1\in X^{(1)}}GW_2^2(k(x_1),\omega_{x_1}) }\ar[r] & \displaystyle{\bigoplus_{x_2\in X^{(2)}}GW_1^1(k(x_2),\omega_{x_2}) }\ar[r] & \displaystyle{\bigoplus_{x_3\in X^{(3)}}GW(k(x_1),\omega_{x_1}) }}$$
and
$$\xymatrix@C=1.2em{\displaystyle{\bigoplus_{x_1\in X^{(1)}}K_2(k(x_1)) }\ar[r] & \displaystyle{\bigoplus_{x_2\in X^{(2)}}K_1(k(x_2)) }\ar[r] & \displaystyle{\bigoplus_{x_3\in X^{(3)}}K_0(k(x_1)) }. }$$

To conclude, we prove that $E(3)_{\infty}^{3,-1}=0$. It suffices to show that the cokernel of the homomorphism 
$$\xymatrix{\displaystyle{\bigoplus_{x_2\in X^{(2)}} GW_2^1(k(x_2),\omega_{x_2})}\ar[r] & \displaystyle{\bigoplus_{x_3\in X^{(3)}} GW_1^0(k(x_3),\omega_{x_3})}}$$
is trivial. Lemma \ref{gw3} yields a commutative diagram
$$\xymatrix{\displaystyle{\bigoplus_{x_2\in X^{(2)}} GW_2^1(k(x_2),\omega_{x_2})}\ar[r]\ar[d]_-f & \displaystyle{\bigoplus_{x_3\in X^{(3)}} GW_1^0(k(x_3),\omega_{x_3})}\ar[d]^-f\\
\displaystyle{\bigoplus_{x_2\in X^{(2)}} \{2\}K_2(k(x_2))}\ar[r] & \displaystyle{\bigoplus_{x_3\in X^{(3)}} \{2\}K_1(k(x_3))}   }$$
in which the left vertical map is surjective and the right vertical map is an isomorphism. Hence both sequences have the same cokernel. 

For any field $F$ and any integer $n\in\N$, define a homomorphism $g_n:K_n(F)/2\to \{2\}K_{n+1}(F)$ by $\alpha\mapsto \{-1\}\cdot \alpha$. It is clear that $g_0$ is an isomorphism, and $g_1$ is surjective by \cite{Suslin87}.

Using the definition of the residue homomorphisms, it is straightforward to check that the diagram
$$\xymatrix{\displaystyle{\bigoplus_{x_2\in X^{(2)}} K_1(k(x_2))/2}\ar[r]\ar[d]_-{\sum g_1} & \displaystyle{\bigoplus_{x_3\in X^{(3)}} K_0(k(x_3))/2}\ar[d]^-{\sum g_0}\\
\displaystyle{\bigoplus_{x_2\in X^{(2)}} \{2\}K_2(k(x_2))}\ar[r] & \displaystyle{\bigoplus_{x_3\in X^{(3)}} \{2\}K_1(k(x_3))}}$$
commutes and therefore the cokernels of the rows are isomorphic. The cokernel of the top homomorphism is $CH^3(X)/2$ which is trivial by \cite[Lemma 1.2]{Colliot}. The result follows.
\end{proof}


\section{Divisibility of $W_E$}\label{div}

In this section, we prove that $W_E(R)$ is divisible for a smooth algebra $R$ of dimension $3$ over an algebraically closed field. Set $X=\spec R$. In view of Proposition \ref{presentation}, it suffices to prove that $H^2(X,K_3)$ is divisible. The idea is to use Bloch-Kato (for $K_2$ and $K_1$) and the Bloch-Ogus spectral sequence.

\begin{prop}\label{h2k3}
Let $X$ be a smooth affine threefold over an algebraically closed field $k$. 
Then $H^2(X,K_3)$ is divisible prime to $\cha k$. 
\end{prop}

\begin{proof}
Let $l$ be a prime number different from $\cha k$. Consider the exact sequences of sheaves (see \cite[proof of Corollary 1.11]{Bloch} for instance)
$$\xymatrix{0\ar[r] & \{l^n\}K_3\ar[r] & K_3\ar[r] & l^nK_3\ar[r] & 0,}$$
and 
$$\xymatrix{0\ar[r] & l^nK_3\ar[r] & K_3\ar[r] & K_3/l^nK_3\ar[r] & 0.}$$ 
Observe that the Gersten resolutions of $\{l^n\}K_3$ and $K_3/l^nK_3$ are flasque resolutions of the corresponding sheaves (see \cite{Grayson85} for instance), and thus we can use these to compute the corresponding cohomology groups.

Since $K_0(F)=\Z$ for any field $F$, we get $H^3(X,\{l^n\}K_3)=0$. This yields the following diagram where the line and the column are exact:
$$\xymatrix{H^2(X,K_3)\ar@{-->}[rd]_{l^n}\ar[r] & H^2(X,l^nK_3)\ar[r]\ar[d] & 0 \\
& H^2(X,K_3)\ar[d]&  \\
& H^2(X,K_3/l^nK_3)&  }$$
It suffices therefore to prove that $H^2(X,K_3/l^nK_3)=0$ to conclude the divisibility by $l^n$.

For any $q,m\in\N$, let $\mathcal H^q(m)$ be the sheaf associated to the presheaf 
$$U\mapsto H^q_{et}(U,\mu_{l^n}^{\otimes m}).$$ 
The Bloch-Ogus spectral sequence (\cite{Blochogus}) converges to the \'etale cohomology groups $H^{*}_{et}(X,\mu_{l^n}^{\otimes m})$ and its groups at page $2$ are the groups $H^p_{Zar}(X,\mathcal H^q(m))$. These are computed via the Gersten complex
$$\xymatrix{H^q(k(X),\mul nm)\ar[r]^-{d_0} & \displaystyle{\bigoplus_{x_1\in X^{(1)}} H^{q-1}(k(x_1),\mul n{m-1})}\ar[r] & \ldots             }$$
Using again that $cd(k(x_p))\leq 3-p$ for any $x_p\in X^{(p)}$ (\cite[\S 4.2, Proposition 11]{Serre}), we get $H^p_{Zar}(X,\mathcal H^q(m))=0$ for any $q\geq 4$. Since $X$ is affine, this implies $H^2(X,\mathcal H^3(m))=H^5_{et}(X,\mu_{l^n}^{\otimes m})=0$ by \cite[Chapter VI, Theorem 7.2]{Milne}. 

Now there is a commutative diagram (\cite[Theorem 2.3]{Bloch})
$$\xymatrix{\displaystyle{\bigoplus_{x_1\in X^{(1)}} K_2(k(x_1))/l^n}\ar[r]\ar[d] & \displaystyle{\bigoplus_{x_2\in X^{(2)}} K_1(k(x_2))/l^n}\ar[r]\ar[d] & \displaystyle{\bigoplus_{x_3\in X^{(3)}} K_0(k(x_3))/l^n}\ar[d]   \\
\displaystyle{\bigoplus_{x_1\in X^{(1)}} H^2(k(x_1),\mu_{l^n}^{\otimes 2})}\ar[r] & \displaystyle{\bigoplus_{x_2\in X^{(2)}} H^1(k(x_2),\mu_{l^n})}\ar[r] & \displaystyle{\bigoplus_{x_3\in X^{(3)}} H^0(k(x_3),\Z/l^n)}     }$$
whose vertical maps are isomorphisms (\cite{Merkusus}). This gives 
$$H^2(X,K_3/l^n)= H^2(X,\mathcal H^3(3))=0.$$ 
The result follows.
\end{proof}


\section{The main theorem}\label{principal}

In this section, we prove that any stably free module $P$ of rank $(d - 1)$ over a 
$d$-dimensional normal affine algebra $R$ over an algebraically closed field $k$, for 
which $\gcd ((d-1)!, \cha k) = 1$, is indeed free. 

Recall first that a unimodular row of length $n\geq 2$ over $R$ is a row $a=(a_1,\ldots,a_n)$ with $a_i\in R$ such that there exist $b_1,\ldots,b_n\in R$ with $\sum a_ib_i=1$. We denote by $Um_n(R)$ the set of unimodular rows of length $n$ and we consider it as a pointed set with base point $e_1:=(1,0,\ldots,0)$. If $M\in GL_n(R)$ and $a\in Um_n(R)$, then $aM$ is also unimodular, and in this way $GL_n(R)$ 
acts on $Um_n(R)$. If $a$ is any unimodular row, the exact sequence
$$\xymatrix{0\ar[r] & P(a)\ar[r] & R^n\ar[r]^-a & R\ar[r] & 0}$$
yields a projective module $P(a)$ which is free if and only if there exists $M\in GL_n(R)$ such that $aM=e_1$. It is clear that any projective module $P$ such that $P\oplus R\simeq R^n$ comes from a unimodular row, and it suffices therefore to understand the (pointed) set $Um_n(R)/GL_n(R)$ to understand stably free modules of rank $n-1$ satisfying $P\oplus R\simeq R^n$. 

Of course, any subgroup of $GL_n(R)$ acts on $Um_n(R)$, and one is classically interested in the (pointed) orbit set $Um_n(R)/E_n(R)$ which classifies unimodular rows up to elementary homotopies. An orbit in $Um_n(R)/E_n(R)$ is usually called an elementary orbit.

We now recall a lemma of L. N. Vaserstein (\cite[Corollary 2]{Vstein4}).

\begin{lem} \label{vstein}
Let $(a_1, a_2, a_3) \in Um_3(R)$, $u \in R$, and 
$Ra_1 + Ru = R$. Then the rows $(a_1, a_2, a_3)$ and $(a_1, ua_2, ua_3)$ 
are in the same elementary orbit. 
\end{lem}

Consequently, by Whitehead's lemma the rows 
$(a_1, ua_2, ua_3)$ and $(a_1, a_2, u^2a_3)$ can be seen to be in the 
same elementary orbit. In particular, if $-1$ is a square in $R$, then 
$(-a_1, a_2, a_3)$ and $(a_1, a_2, a_3)$ are in the same elementary orbit. 

Given $v = (v_0, v_1, v_2)$, $w = (w_0, w_1, w_2) \in Um_3(R)$, with 
$v\cdot w^t = 1$, we shall denote by 
$V(v, w)$ the class in $W_E(R)$ of the $4 \times 4$ skew-symmetric matrix  
$$\begin{pmatrix} 0 & v_0 & v_1 & v_2\\
                                 -v_0 &   0 & w_2 & -w_1\\
                                 -v_1 &-w_2 & 0 & w_0 \\
                                 -v_2 &w_1 &-w_0 & 0 \end{pmatrix}
$$
of Pfaffian $1$.
For simplicity, we may just write $V(v)$, as the class of $V(v, w)$ does not 
depend on the choice of $w$ such that $v\cdot w^t = 1$. Associating $V(v)$ to $v$ yields a well defined map $V:Um_3(R)/E_3(R)\to W_E(R)$ called the Vaserstein symbol (\cite[\S 5]{SV}).

\begin{thm}\label{Rao94}
Let $R$ be a smooth affine algebra over a field $k$ of cohomological dimension~$\leq 1$ which is perfect if the characteristic is $2$ or $3$. Then 
the Vaserstein symbol 
$$V: Um_3(R)/E_3(R)\to W_E(R)$$ 
is a bijection. 
\end{thm}

\begin{proof}
See \cite[Corollary 3.5]{Rao94}.
\end{proof}

This bijection gives rise to an abelian group structure on the set $Um_3(R)/E_3(R)$, when $R$ is as above. We recall Vaserstein's rule in $W_E(R)$ (see \cite[Theorem 5.2 (iii)]{SV}). If $(a,b,c)$ and $(a,x,y)$ are in $Um_3(R)/E_3(R)$, then 
$$V(a, b, c) \perp V(a, x, y) = V\left(a, (b, c)\begin{pmatrix} x & y\\ -y' & x\end{pmatrix}\right)$$
for any $x^\prime,y^\prime$ such that $xx' + yy' = 1$ modulo $(a)$. 

\begin{lem}\label{um}
Let $R$ be a ring. The following formulas hold in $W_E(R)$:
\begin{enumerate}[(i)]
\item $V(a, b, c)^{-1} = V(-a', b, c)$, where $a^\prime$ is an inverse of $a$ modulo $(b,c)$.
\item $V(ab^2, c, d)=V(a, c, d) \perp V(b^2, c, d)$.
\end{enumerate}
\end{lem}

\begin{proof}
Since Vaserstein's rule holds in $W_E(R)$, it suffices to follow the proof of \cite[Lemma 3.5(iii), Lemma 3.5(v)]{vdk}.
\end{proof}

We now recall the \emph{Antipodal Lemma} of Rao in \cite[Lemma 1.3.1]{Rao87}.

\begin{lem}\label{antipodal}
Let $R$ be a ring and $(a,b,c)\in Um_3(R)$. Suppose that $(a,b,c)$ and $(-a,b,c)$ are in the same elementary orbit. Then for any $n\in\N$, we have 
$$nV(a,b,c)=V(a^n,b,c)$$
in $W_E(R)$.
\end{lem}

\begin{proof}
The proof of \cite[Lemma 1.3.1]{Rao87} applies mutatis mutandis using the formulas of Lemma \ref{um}.
\end{proof}

We now state and prove our main theorem:

\begin{thm}\label{faselswr}
Let $R$ be a $d$-dimensional normal affine algebra over an algebraically closed field $k$ such that $\gcd((d-1)!, \cha k)=1$. If $d=3$, suppose moreover that $R$ is smooth. Then every stably free $R$-module $P$ of rank $d-1$ is free.
\end{thm}

\begin{proof}
Let $P$ be a stably free module of rank $d-1$. Since the result is clear when $d\leq 2$, we assume that $d\geq 3$. Using Suslin's cancellation theorem \cite[Theorem 1]{Suslin1}, we can suppose that there is an isomorphism $P\oplus R\simeq R^d$, and therefore that $P$ is given by a unimodular row $(a_1,\ldots,a_d)$. In view of \cite[Theorem 2]{Suslin1}, to prove that $P$ is free it suffices to show that there exists a unimodular row $(b_1,\ldots,b_d)$ such that $(a_1,\ldots,a_d)=(b_1^{(d-1)!},\ldots,b_d)$ in $Um_d(R)/E_d(R)$. 

Suppose that $d\geq 4$. Let $J$ be the ideal of the singular locus of $R$. Since $R$ is normal, $J$ has height at least $2$ and $\dim (R/J)\leq d-2$. It follows from \cite[Theorem 3.5]{Bass68} that $Um_d(R/J)=e_1E_d(R/J)$ and we can therefore assume, performing elementary operations if necessary, that $a_d\equiv 1\pmod J$ and $a_1,\ldots,a_{d-1}\in J$. Using now Swan's Bertini theorem \cite[Theorem 1.5]{Swan}, we can perform elementary operations on $(a_1,\ldots,a_d)$ such that $B:=R/(a_4, \ldots , a_d)$ is either empty, either a non-singular threefold outside the singular locus of $R$. In the first case, the row $(a_4, \ldots , a_d)$ is unimodular, and therefore the row $(a_1,\ldots,a_d)$ is completable in an elementary matrix. Thus we can restrict to the second case. In this situation, we see that $B$ is actually smooth since $a_d\equiv 1 \pmod J$.

Given a unimodular row $(\overline a,\overline b,\overline c)$ on $B$, we can choose lifts $a,b,c\in R$ and consider the unimodular row $(a,b,c,a_4,\ldots,a_d)$ on $R$. It is straightforward to check that this gives a well-defined map
$$Um_3(B)/E_3(B)\to Um_d(R)/E_d(R),$$
showing that $(a_1,\ldots,a_d)$ comes from the unimodular row $(\overline a_1,\overline a_2,\overline a_3)$ on $B$.
We are thus reduced to the case where $R$ is the affine algebra of a smooth threefold. By Theorem \ref{Rao94}, the set $Um_3(R)/E_3(R)$ is in bijection with $W_E(R)$ and is thus endowed with the structure of an abelian group. Since $-1$ is a square in $k$, Lemma \ref{antipodal} shows that $n\cdot (a_1,a_2,a_3)=(a_1^n,a_2,a_3)$ in $Um_3(R)/E_3(R)$ for any $n\in\N$. Now Propositions \ref{presentation} and \ref{h2k3} show that $Um_3(R)/E_3(R)$ is a divisible group prime to the characteristic of $k$. Since $\gcd((d-1)!, \cha k)=1$, there exists a unimodular row $(b_1,b_2,b_3)\in Um_3(R)$ such that 
$$(a_1,a_2,a_3)=(d-1)!\cdot (b_1,b_2,b_3)=(b_1^{(d-1)!},b_2,b_3)$$
in $Um_3(R)/E_3(R)$. The result follows.
\end{proof}

\begin{rem}\label{faseldd}
The proof of the theorem actually gives a bit more. Let $R$ be an affine algebra of dimension $d\geq 4$ over an algebraically closed field $k$ such that $\gcd((d-1)!, \cha k)=1$. Let $J$ be the ideal of the singular locus of $R$, and let $(a_1,\ldots,a_d)$ be a unimodular row of rank $d$ which is in the elementary orbit of $e_1$ in $Um_d(R/J)/E_d(R/J)$. Then the projective module associated to $(a_1,\ldots,a_d)$ is free. To see this, observe that we can use \cite[Theorem 2]{Roit1} to reduce to the case of a reduced ring and then apply the proof of the theorem.
\end{rem}

To conclude, we derive from the above theorem some stability results for the group $K_1(R)$ where $R$ is a smooth affine algebra over an algebraically closed field. 

\begin{cor} \label{injstab}
Let $R$ be a smooth algebra of dimension $d\geq 3$ over an algebraically 
closed field $k$. Assume that $d!k = k$. Then $SL_d(R)\cap E_{d+1}(R)=E_d(R)$.
\end{cor}

\begin{proof}
Let $\sigma \in SL_d(R)\cap E_{d+1}(R)$. In view of T. Vorst's theorem in \cite{Vorst}, it suffices to show that there exists a matrix $N(X)\in SL_d(R[X])$ with the property that $N(0)=Id$ and $N(1)=\sigma$. 

Since $1\perp \sigma\in E_{d+1}(R)$, there exists $\alpha(X)\in GL_{d+1}(R[X])$ with $\alpha(0)=Id$ and $\alpha(1)=1\perp \sigma$. It follows that the unimodular row $v(X):=e_1\alpha(X)\in Um_{d+1}(R[X])$ is congruent to $e_1$ modulo $f$, where $f:=X^2-X$. We now prove that there exists a matrix $M(X)\in GL_{d+1}(R[X])$ congruent to $Id$ modulo $f$ satisfying the equality $e_1M(X)=e_1\alpha(X)$. This will prove the corollary since these properties imply that
$$\alpha(X)M(X)^{-1}=\begin{pmatrix} 1 & 0\\ \star & N(X)\end{pmatrix}$$
with $N(X)\in GL_d(R[X])$ such that $N(0)=Id$ and $N(1)=\sigma$.

To find $M(X)$, we argue as in \cite[Proposition 3.3]{Rao94}. Set $B = R[X, T]/(T^2 - Tf)$
and write $e_1 \alpha(X) = e_1 +(X^2-X)w(X)$, with $w(X) \in R[X]^{d+1}$. Observe that $u(T):= e_1 + Tw(X)$ is a unimodular row of length $d+1$ over $B$ with $u(f)=v(X)$ and $u(0)=e_1$. Since $R$ is smooth, it follows that the ideal of the singular locus of $B$ is precisely $(T,f)$. Now $u(T)$ is congruent to $e_1$ modulo $T$ and we can use Remark \ref{faseldd} to find $\varepsilon(T) \in E_{d+1}(B)$ such that $u(T)\varepsilon(T)$ is of the form $(b_1^{d!},\ldots,b_{d+1})$. By \cite[Proposition 3.3]{Rao94} or \cite[Proposition 5.4]{Rao09}, there is a $\delta(T) \in SL_{d+1}(B)$ congruent to $Id$ modulo $f$ such that $u(T)\varepsilon(T) = e_1\delta(T)$. Finally, we set $M(X):=\beta(0)\beta(f)^{-1}$ where $\beta(T) = \varepsilon(T)\delta(T)^{-1}$. It is straightforward to check that $M(X)$ satisfies the conditions stated above, and therefore the result is proved.
\end{proof}

\begin{cor} \label{finite}
Let $R$ be a smooth algebra over the algebraic closure $k$ of 
a finite  field of dimension $d\geq 3$. Assume that $d!k = k$. Then the 
natural map $SL_d(R)/E_d(R) \to SK_1(R)$ is an isomorphism. 
\end{cor}

\begin{proof}
Using \cite[Corollary 17.3]{SV}, we see that $Um_{n}(R) = e_1E_n(R)$ for $n\geq d$. It follows that $SL_d(R)/E_d(R) \to SK_1(R)$ is onto. We conclude using Corollary \ref{injstab}.
\end{proof}

\begin{cor}\label{vsym4}
Let $R$ be a smooth algebra of dimension $d \leq 4$ over an algebraically closed field $k$. Assume that $d!k = k$. Then 
the Vaserstein symbol 
$$V: Um_3(R)/E_3(R) \to W_E(R)$$
is an isomorphism. 
\end{cor}

\begin{proof} L.N. Vaserstein in \cite[Theorem 5.2(c)]{SV} has shown that $V$ is onto. Apply \cite[Lemma 5.1]{SV} to conclude that $V$ is injective as $SL_4(R) \cap E(R) = E_4(R)$ in view of Corollary \ref{injstab}.
\end{proof}


\bibliographystyle{plain}
\bibliography{biblio_ravi.bib}

\noindent
Jean Fasel,\\
Mathematisches Institut der Universit\"at M\"unchen, \\
Theresienstrasse 39,\\
D-80333 M\"unchen, Germany\\
Jean Fasel $<$jean.fasel@gmail.com$>$

\medskip

\noindent
Ravi A. Rao,\\
Tata Institute of Fundamental Research,\\
1, Dr. Homi Bhabha Road, Navy Nagar, \\
Mumbai 400 005, India\\
Ravi A. Rao $<$ravi@math.tifr.res.in$>$

\medskip

\noindent
Richard G. Swan,\\
University of Chicago,\\
Chicago, Illinois, 60637, U.S.A.\\
Richard G. Swan $<$swan@math.uchicago.edu$>$

\end{document}